\documentclass[10pt, english]{elsarticle}
\usepackage{amsthm}
\usepackage{amsmath}
\usepackage{latexsym, amssymb}
\usepackage{txfonts}
\usepackage{mathtools}
\usepackage{color}
\usepackage{babel}
\usepackage[all]{xy}
\usepackage[capitalise]{cleveref}

\newtheorem{thm}{Theorem}[section] 

\newtheorem{cor}[thm]{Corollary}

\newtheorem{lem}[thm]{Lemma}
\newtheorem{prop}[thm]{Proposition}

\newtheorem{ques}[thm]{Question}

\theoremstyle{definition}

\newtheorem{defn}[thm]{Definition}

\newcommand\operA[2]{{\if!#2!\operatorname{#1}\else{\operatorname{#1}_{#2}^{\phantom{I}}}\fi}} 

\newcommand{\Trace}[1][]{\if!#1!\operatorname{Tr}\else{\operatorname{Tr}_{#1}^{\phantom{I}}}\fi} 

\long\def\forget#1\forgotten{{}} %

\def\({\left(}
\def\){\right)}

\newcommand\LAY[3][]{{\begin{array}{c}\mbox{#2} \if#1!{}\else{+}\fi \\ \mbox{#3}\end{array}}}

\makeatletter
\newcommand{\bigperp}{%
  \mathop{\mathpalette\bigp@rp\relax}%
  \displaylimits
}

\newcommand{\bigp@rp}[2]{%
  \vcenter{
    \m@th\hbox{\scalebox{\ifx#1\displaystyle2.1\else1.5\fi}{$#1\perp$}}
  }%
}
\makeatother

\newcommand{\qf}[1]{\mbox{$\langle #1\rangle $}}

\renewcommand{\geq}{\geqslant}
\renewcommand{\leq}{\leqslant}

\makeatletter
\def\ps@pprintTitle{%
 \let\@oddhead\@empty
 \let\@evenhead\@empty
 \def\@oddfoot{\centerline{\thepage}}%
 \let\@evenfoot\@oddfoot}
\makeatother

\newif\iffurther
\furtherfalse

\begin{document}
\begin{frontmatter}

\title{Triple Linkage of Quadratic Pfister Forms}

\author{Adam Chapman}
\ead{adam1chapman@yahoo.com}
\address{Department of Computer Science, Tel-Hai Academic College, Upper Galilee, 12208 Israel}
\author{Andrew Dolphin}
\ead{Andrew.Dolphin@uantwerpen.be}
\address{Department of Mathematics, Ghent University,  Ghent, Belgium}
\author{David B. Leep}
\ead{leep@uky.edu}
\address{Department of Mathematics, University of Kentucky, Lexington, KY 40506}

\begin{abstract}
Given a field $F$ of characteristic 2, we prove that if every three quadratic $n$-fold Pfister forms have a common quadratic $(n-1)$-fold Pfister factor then $I_q^{n+1} F=0$.
As a result, we obtain that if every three quaternion algebras over $F$ share a common maximal subfield then $u(F)$ is either $0,2$ or $4$.
We also prove that if $F$ is a nonreal field with $\operatorname{char}(F) \neq 2$ and 
$u(F)=4$,  then every three quaternion algebras share a common maximal subfield.
\end{abstract}

\end{frontmatter}

\section{Introduction}

We say that a set of quadratic $n$-fold Pfister forms is linked if there exists a quadratic or bilinear $(n-1)$-fold Pfister form which is a common factor to all the forms in this set.
By the natural identification of quaternion algebras with their norm forms which are quadratic 2-fold Pfister forms, a set of quaternion algebras is linked if they share a common maximal subfield.

A maximal subfield $K$ of a quaternion algebra over $F$ is a quadratic field extension of $F$. When $\operatorname{char}(F)=2$, $K/F$ can be either separable or inseparable, and one can refine the definition of linkage accordingly: a set of quaternion algebras is separably  (inseparably) linked if they share a common separable (inseparable) quadratic field extension of $F$. 
It was observed in \citep{Draxl:1975} that inseparable linkage for pairs of quaternion algebras implies separable linkage, and a counter example for the converse was provided in \cite{Lam:2002}. This observation was extended to pairs of Hurwitz algebras in \cite{ElduqueVilla:2005} and pairs of cyclic $p$-algebras of any prime degree in \cite{Chapman:2015}.

We extend the notion of separable and inseparable linkage of quaternion algebras to arbitrary $n$-fold Pfister forms in the following way:
a set of quadratic $n$-fold Pfister forms are separably (inseparably) linked if there exists a quadratic (bilinear) $(n-1)$-fold Pfister form which is a common factor to all the forms in the set.
It can be easily concluded from the fact mentioned above (as can be seen in \cite{Faivre:thesis}) that if two quadratic $n$-fold Pfister forms $\varphi_1$ and $\varphi_2$ satisfy $\varphi_1=B \otimes \pi_1$ and $\varphi_2=B \otimes \pi_2$ for some bilinear $(n-1)$-fold Pfister form $B$ and quadratic 1-fold Pfister forms $\pi_1$ and $\pi_2$, then $\varphi_1$ and $\varphi_2$ are separably linked.

It was proven in \cite[Main Theorem]{ElmanLam:1973} that if $F$ is a field of $\operatorname{char}(F) \neq 2$ and every two quaternion algebras over $F$ are linked then the possible values $u(F)$ can take are $0,1,2,4$ and 8, where $u(F)$ is defined to be the supremum of the  dimensions of  nonsingular anisotropic quadratic forms over $F$ of a finite order in $W_q F$. 
For fields $F$ of characteristic $2$, it was shown in
 \cite[Theorem 3.1]{Baeza:1982} that every two quaternion algebras over $F$ share a quadratic inseparable field extension of $F$ if and only if $u(F) \leq 4$, and in 
 \cite[Corollary 5.2]{ChapmanDolphin} that if every two quaternion algebras over $F$ share a maximal subfield then $u(F)$ is either $0,2,4$ or $8$.

In \cite{becher}, the case of linkage of three bilinear $n$-fold Pfister forms in any characteristic was studied.
It was shown in \cite[Theorem 5.1]{becher} that if $F$ is a nonreal field and every three bilinear $n$-fold Pfister forms are linked then $I^{n+1} F=0$.
When $\operatorname{char}(F) \neq 2$, by the natural identification of quadratic forms with their underlying symmetric bilinear forms, this implies that if every three $n$-fold Pfister forms are linked then $I_q^{n+1} F=0$.
This was used to study the Hasse number of $F$. The Hasse number, denoted $\tilde{u}(F)$, is the supremum of the dimensions of anisotropic totally indefinite quadratic forms over $F$. Note that if $F$ is nonreal, then $\tilde{u}(F)=u(F)$. 
It was shown in \cite[Corollary 5.8]{becher} that if $F$ is a field with $\operatorname{char}(F) \neq 2$ and every three quaternion algebras over $F$ are linked then $\tilde{u}(F) \leq 4$.
The question of whether the converse held, that is whether $\tilde{u}(F) \leq 4$ implies 
that all triples of quaternion algebras over $F$ are linked, appeared in a preliminary version of 
 of \cite{becher}.

%


In this paper we prove that if $\operatorname{char}(F)=2$ and every three quadratic $n$-fold Pfister forms in $F$ are linked then $I_q^{n+1} F=0$.
We conclude that if every three quaternion algebras over $F$ are linked then $u(F) \leq 4$. In the last section we show that if  $F$ is a nonreal field with $\operatorname{char}(F) \neq 2$ and ${u}(F)=4$ then every three quaternion algebras over $F$ are linked.
Note that   we shared this last result with the author of \cite{becher} who included a similar proof in the final version of that paper, acknowledging our contribution.


\section{Bilinear and Quadratic Pfister Forms}

Let $F$ be a field with $\operatorname{char}(F)=2$. We recall what we need from the algebraic theory of quadratic forms. For general reference see \cite[Chapters 1 and 2]{EKM}.

Let $V$ be an $n$-dimensional $F$-vector space. A symmetric bilinear form on $V$ is a map $B : V \times V \rightarrow F$ satisfying $B(v,w)=B(w,v)$, $B(cv,w)=c B(v,w)$ and $B(v+w,t)=B(v,t)+B(w,t)$ for all $v,w,t \in V$ and $c \in F$.
A symmetric bilinear form $B$ is degenerate if there exists a vector $v \in V\setminus\{0\}$ such that $B(v,w)=0$ for all $w \in V$.
If such a vector does not exist, we say that $B$ is nondegenerate.
Two symmetric bilinear  forms $B: V\times V \rightarrow F$ and $B' : W \times W\rightarrow F$ are isometric if there exists an isomorphism $M : V \rightarrow W$ such that $B(v,v')=B'(Mv,Mv')$ for all $v,v' \in V$.

A quadratic form over $F$ is a map $\varphi : V \rightarrow F$ 
such that $\varphi(av)= a^2\varphi(v)$ for all $a\in F$ and $v\in V$ and the map defined by  $B_\varphi(v,w)=\varphi(v+w)-\varphi(v)-\varphi(w)$ for all $v,w\in V$ is a bilinear form on $V$. The bilinear form $B_\varphi$ is called the polar form of $\varphi$ and is clearly symmetric.  
  Two quadratic forms $\varphi : V \rightarrow F$ and $\psi : W \rightarrow F$ are isometric if there exists an isomorphism $M : V \rightarrow W$ such that $\varphi(v)=\psi(Mv)$ for all $v \in V$.
We are interested in the isometry classes of quadratic forms, so when we write $\varphi=\psi$ we actually mean that they are isometric.

%
%

We say that $\varphi$ is singular if $B_\varphi$ is degenerate, and that $\varphi$ is nonsingular if $B_\varphi$ is nondegenerate.
Every nonsingular form $\varphi$ is even dimensional and can be written as 
$$\varphi=[\alpha_1,\beta_1] \perp \dots \perp [\alpha_n,\beta_n]$$
for some $\alpha_1,\dots,\beta_n \in F$, where $[\alpha,\beta]$ denotes the two-dimensional quadratic form $\psi(x,y)=\alpha x^2+xy+\beta y^2$ and $\perp$ denotes  the orthogonal sum of quadratic forms.

We say that a quadratic form $\varphi : V \rightarrow F$ is isotropic if there exists a vector $v \in V\setminus\{0\}$ such that $\varphi(v)=0$.
If no such vector exists, we say that $\varphi$ is anisotropic.
The unique nonsingular two-dimensional isotropic quadratic form is $\varmathbb{H}=[0,0]$, called ``the hyperbolic plane".
A hyperbolic form is an orthogonal sum of hyperbolic planes.
We say that two nonsingular quadratic forms are Witt equivalent if their orthogonal sum is a hyperbolic form.

We denote by $\langle \alpha_1,\dots,\alpha_n \rangle$ the diagonal bilinear form given by  
$(x,y)\mapsto \sum_{i=1}^n \alpha_ix_iy_i$. 
Given two symmetric bilinear forms $B_1:V\times V\rightarrow F$ and $B_2:W\times W\rightarrow F$,   the tensor product of $B_1$ and $B_2$ denoted $B_1\otimes B_2$  
 is the unique $F$-bilinear map $B_1\otimes B_2:(V\otimes_F W)\times (V\otimes_F W)\rightarrow F$ such that 
$(B_1\otimes B_2)\left( (v_1\otimes w_1), (v_2\otimes w_2)\right) =B_1(v_1,v_2)\cdot B_2(w_1,w_2) $
for all $w_1,w_2\in W, v_1,v_2\in V$.

A bilinear $n$-fold Pfister form over $F$ is a symmetric bilinear form isometric to a bilinear form
$$\langle 1,\alpha_1\rangle \otimes \langle 1,\alpha_2\rangle \otimes \dots \otimes \langle 1,\alpha_n\rangle$$
for some $\alpha_1,\alpha_2,\dots,\alpha_n \in F^\times$.
We denote such a form by $\langle \langle \alpha_1,\alpha_2,\dots,\alpha_n \rangle \rangle$. By convention, the bilinear 0-fold Pfister form is $\langle 1 \rangle$.


  Let $B:V\times V\rightarrow F$ be a symmetric bilinear form over $F$ and $\varphi:W\rightarrow F$ be a quadratic form over $F$. We may define a  quadratic form $B\otimes \varphi:V\otimes_F W\rightarrow F$ determined by the rule that 
$( B\otimes \varphi) (v\otimes w)=  B(v,v) \cdot \varphi(w)$
for all $w\in W, v\in V$. We call this quadratic form  the tensor product of $B$ and $\varphi$. 
A quadratic $n$-fold Pfister form over $F$ is a tensor product of a bilinear $(n-1)$-fold Pfister form  $\langle \langle \alpha_1,\alpha_2,\dots,\alpha_{n-1} \rangle \rangle$ and a two-dimensional quadratic form  $[1,\beta]$ for some $\beta \in F$.
We denote such a form by $\langle \langle \alpha_1,\dots,\alpha_{n-1},\beta]]$. Quadratic $n$-fold Pfister forms are isotropic if and only if they are hyperbolic (see  \cite[(9.10)]{EKM}).

The set of Witt equivalence classes of nonsingular quadratic forms is an abelian group with $\perp$ as the binary group operation and the class of $\varmathbb{H}$ as the zero element. We denote this group by $I_q F$ or $I_q^1 F$.
This group is generated by scalar multiples of quadratic 1-fold Pfister forms.
Let $I_q^n F$ denote the subgroup generated by scalar multiples of quadratic $n$-fold Pfister forms.

A quaternion algebra over $F$ is an $F$-algebra of the form $[\alpha,\beta)_{2,F}=F \langle x,y: x^2+x=\alpha, y^2=\beta, y x y^{-1}=x+1 \rangle$ for some $\alpha \in F$ and $\beta \in F^\times$. Its norm form is the quadratic 2-fold Pfister form $\langle \langle \beta,\alpha]]$. This matching provides a 1-to-1 correspondence between quaternion $F$-algebras and quadratic 2-fold Pfister forms over $F$.
Therefore, the notions of separable and inseparable linkage translate naturally from quaternion algebras to $2$-fold Pfister forms.
We extend these notions of linkage to quadratic $n$-fold Pfister forms for any $n \geq 2$ in the following way: a set $S=\{ \varphi_1,\dots,\varphi_m\}$ of $m$ quadratic $n$-fold Pfister forms is separably linked if there exist a quadratic $(n-1)$-fold Pfister form $\phi$ and bilinear 1-fold Pfister forms $B_1,\dots,B_m$ such that $\varphi_i=B_i \otimes \phi$ for all $i \in \{1,\dots,m\}$. The set $S$ is inseparably linked if there exists a bilinear $(n-1)$-fold Pfister form $B$ and quadratic 1-fold Pfister forms $\phi_1,\dots,\phi_m$ such that $\varphi_i=B \otimes \phi_i$ for all $i \in \{1,\dots,m\}$.

We make use of the following well-known results on quaternion algebras, and, in the case of $(1)$,  its reinterpretation  in terms of $2$-fold Pfister forms. 

\begin{lem}[{\cite[VII.1.9]{BO}}]\label{iso} 
For $\alpha\in F$ and $\beta,\gamma\in F^\times$ we have 
\begin{enumerate} 
\item[$(1)$] $[\alpha,\beta)_{2,F}\cong [\alpha^2 +\beta , \beta)_{2,F}$. Further if $\alpha\neq 0$,  then $[\alpha,\beta)_{2,F}\cong [\alpha, \beta\alpha)_{2,F}$.
\item[$(2)$]  $[\alpha,\beta)_{2,F}\otimes_F [\alpha,\gamma)_{2,F}$ is Brauer equivalent to $[\alpha,\beta\gamma)_{2,F}$.
\end{enumerate}
\end{lem}

%

\section{Triple Linkage for Quadratic Pfister Forms}\label{SectionTriple}

In this section we study properties of fields $F$ with $\operatorname{char}(F)=2$ in which every three quadratic $n$-fold Pfister forms are linked for some given integer $n \geq 3$. The case $n=2$ is somewhat different and will  be handled  in the next section.
We make use of the following invariant defined in \cite[Definition 4.1]{ChapmanGilatVishne}:
\begin{defn}\label{Invariant}
Let $n$ be an integer $\geq 2$. Consider a set $S$ of quadratic $n$-fold Pfister forms.
We say $S$ is \textbf{tight} if every element in the subgroup $G$ of $I_q^n F/I_q^{n+1} F$ generated by the Witt classes of the forms in $S$ is represented by a quadratic  $n$-fold  Pfister form.
For such a finite tight set, we define $\Sigma_S$ to be the Witt Class of the orthogonal sum of the quadratic $n$-fold Pfister representatives of the elements in $G$.
By the identification of quaternion algebras with their norm forms, the notion of tightness and the associated invariant also apply for sets of quaternion algebras in the Brauer group.
\end{defn}

A similar invariant was studied in \cite{Sivatski:2014} when $\operatorname{char}(F) \neq 2$ and $|S|=3$.

\begin{lem}\label{lem:tight}
Let $n$ and $k$ be  integers $\geq 2$ and let $S=\{\varphi_1,\dots,\varphi_k\}$ be a tight set of $k$ quadratic $n$-fold Pfister forms over a field $F$  with $\operatorname{char}(F)=2$.
\begin{enumerate}
\item If $S$ is separably linked then $\Sigma_S$ is the Witt class of a quadratic $(n+k-1)$-fold Pfister form. More precisely, if $\varphi_i=\langle \langle a_i \rangle \rangle \otimes \phi$ for each $i \in \{1,\dots,k\}$, then $\Sigma_S=\langle \langle a_1,\dots,a_k \rangle \rangle \otimes \phi$.
\item If $S$ is inseparably linked then $\Sigma_S$ is the trivial Witt class.
\end{enumerate}
\end{lem}

\begin{proof}
The first statement is a special case of \cite[Proposition 4.3]{ChapmanGilatVishne}. The second statement is an immediate result of \cite[Corollary 3.6 and Proposition 4.2 (2)]{ChapmanGilatVishne}
\end{proof}

\begin{thm}
Let $F$ be a field with $\operatorname{char}(F)=2$.
Then for any integer $n \geq 3$, if every set of three quadratic $n$-fold Pfister forms over $F$ is  either separably or inseparably linked then $I_q^{n+1} F=0$.
\end{thm}

\begin{proof}
Let $n$ be an integer $\geq 3$. Consider a quadratic $(n+1)$-fold Pfister form $\Phi$. Then $\Phi$ can be written as $\langle \langle a,b,c \rangle \rangle \otimes \psi$ where $\psi$ is some quadratic $(n-2)$-fold Pfister form and $a,b,c$ are some elements in $F^\times$.
Consider the forms $\varphi_1=\langle \langle a,b \rangle \rangle \otimes \psi$, $\varphi_2=\langle \langle a,c \rangle \rangle \otimes \psi$ and $\varphi_3=\langle \langle b,c \rangle \rangle \otimes \psi$.
Then in  $I_q^n F/I_q^{n+1} F$ we have 
$$\varphi_1\perp \varphi_2 = \qf{b,ab,c,ac}\otimes \psi=  \langle \langle a,bc \rangle \rangle \otimes \psi\mod I_q^{n+1} F\,.$$
Similarly, 
\begin{eqnarray*}\varphi_1\perp \varphi_3 = \langle \langle b,ac \rangle \rangle \otimes \psi\mod I_q^{n+1} F\,,
\\ \varphi_2\perp \varphi_3 = \langle \langle c,ab \rangle \rangle \otimes \psi\mod I_q^{n+1} F\,,
\end{eqnarray*}
and 
$$ \varphi_1\perp \varphi_2\perp \varphi_3 = \langle \langle ab,ac \rangle \rangle \otimes \psi\mod I_q^{n+1} F\,. $$ 
Therefore $S=\{\varphi_1,\varphi_2,\varphi_3\}$ is a tight triplet. 
 A straight-forward computation then gives that $\Sigma_S$ is Witt equivalent to $\Phi$.

Since $\varphi_1,\varphi_2,\varphi_3$ are quadratic $n$-fold Pfister forms, $S$ is separably or inseparably linked. If they are separably linked then $\Sigma_S$ is a quadratic $(n+2)$-fold Pfister form by \cref{lem:tight} $(1)$.
By the Hauptstatz theorem \cite[23.7]{EKM}, this can happen only when $\Phi$ is hyperbolic.
If $S$ is inseparably linked then $\Sigma_S$ is hyperbolic by \cref{lem:tight}, $(2)$ and so $\Phi$ is hyperbolic as well.
Consequently $I_q^{n+1} F=0$.
\end{proof}

\begin{cor}
For any integer $n \geq 3$, if every set of  three quadratic $n$-fold Pfister forms over $F$  is   separably linked then every set of  two quadratic $n$-fold Pfister forms over $F$ is  inseparably linked.
\end{cor}

\begin{proof}
By the previous theorem $I_q^{n+1}(F)=0$.
By the assumption, every set of  two quadratic $n$-fold Pfister forms over $F$ is separably  linked.
By \cite[Corollary 5.4]{ChapmanGilatVishne}, $I_q^{n+1}(F)=0$ implies that two quadratic $n$-fold Pfister forms over $F$ are separably linked if and only if they are inseparably linked.
Hence every set of  two quadratic  $n$-fold Pfister forms over $F$ are inseparably linked.
\end{proof}

\section{Triple Linkage for Quaternion Algebras}

We again fix $F$ to be a field with $\operatorname{char}(F)=2$.
In this section we complete the picture for the case of quadratic  2-fold Pfister forms over $F$. These forms are in in 1-1 correspondence with quaternion algebras and our proofs are mainly written in terms of these algebras.
The proof in this particular case is somewhat different from the case of higher-fold Pfister forms.

\begin{lem}\label{rightslotsquare} Let $\alpha,\beta,\lambda\in F$ with $\lambda^2\neq \beta\neq 0$.  Then  $[\alpha,\beta)_{2,F}=[\alpha+\lambda^2 \alpha \beta^{-1},\beta+\lambda^2)_{2,F}$.
\end{lem}

\begin{proof}
Write $[\alpha,\beta)_{2,F}=F \langle x,y : x^2+x=\alpha, y^2=\beta, y x y^{-1}=x+1 \rangle$.
Let $w=y+\lambda$ and $z=x+\lambda x y^{-1}$. Note that $w^2=\lambda^2+\beta \in F^\times$, and in particular $w$ is invertible.
We have that $wz-zw=y+\lambda=w$, i.e. $wzw^{-1}=z+1\,.$
Therefore $$[\alpha,\beta)_{2,F}=[z^2+z,w^2)_{2,F}=[\alpha+\lambda^2 \alpha \beta^{-1},\beta+\lambda^2)_{2,F}\,,$$ as required.\end{proof}

\begin{lem}\label{leftoright} Let $\alpha\in F, \beta,\gamma\in F^\times$ with $\beta\neq \alpha^2$.
 Then 
the quaternion algebras $\psi=[\alpha^2+\beta,\gamma)_{2,F}$ and $\phi=[\beta+\alpha^4 \beta^{-1},\beta+\alpha^2)_{2,F}$ are separably linked and  $\Sigma_{\{ \psi, \phi \} }$ is Witt equivalent to $\langle \langle \gamma,\beta,\alpha]]$ .
\end{lem}

\begin{proof}
Let $\phi'= [\alpha^2+\beta,\beta)_{2,F}$. Then 
by   \cref{rightslotsquare} (taking $\lambda=\alpha$) we have  
$$\phi' = [\alpha^2+\beta+\alpha^2(\alpha^2+\beta)\beta^{-1},\beta+\alpha^2)_{2,F} = [\beta+\alpha^4\beta^{-1},\beta+\alpha^2)_{2,F}= \phi\,.$$
Hence $\psi$ and $\phi$ are separably  linked. Furthermore, by \cref{iso}, $(2)$ we have that 
$[\alpha^2+\beta,\beta\gamma)_{2,F}$ is the quaternion representative of $\psi\otimes_F \phi$ in the Brauer group. Hence $S=\{\psi,\phi\}$ is a tight pair, and by Lemma \ref{lem:tight} (1) and  \cref{iso} $(1)$ we have  
$\Sigma_S=\langle \langle \gamma,\beta,\alpha^2+\beta]]=\langle \langle \gamma,\beta,\alpha^2+\beta]]$.
%
%
%
%
\end{proof}

\begin{prop}\label{Witteq} Let $\alpha\in F, \beta,\gamma\in F^\times$ with $\beta\neq \alpha^2$.
Then the quaternion algebras  $\psi=[\beta+\alpha^4 \beta^{-1},\gamma)_{2,F}$, $\phi=[\beta+\alpha^4 \beta^{-1},\gamma(\alpha^2+\beta))_{2,F}$ and $\pi=[\alpha^2+\beta,\gamma)_{2,F}$ form a tight triplet and   $\Sigma_{\{ \psi,\phi,\pi \} }$ is Witt equivalent to $\langle \langle \gamma,\beta,\alpha]]$.
\end{prop}

\begin{proof}
The algebra $\pi$ is evidently inseparably, and hence separably (by \cite{Lam:2002}), linked to $\psi$ and to $\phi$ using \cref{iso}, $(1)$. Further, $\psi$ and $\phi$ are separably linked.
By \cref{iso} $(2)$, we have that  $\xi= [\beta+\alpha^4\beta^{-1},\alpha^2+\beta)_{2,F}$ is the quaternion representative of $\psi \otimes_F \phi$  in the Brauer group.
By  \cref{leftoright},  $\xi$  is separably linked to $\pi$ as well.
Using   \cref{iso} $(2)$ we see that the tensor product of any two separably linked quaternion algebras is  represented by a quaternion algebra in the Brauer group.
We conclude that $S=\{ \psi,\phi,\pi \}$ is a tight triplet of quaternion algebras.

Since the pairs $\psi, \phi$  and $\psi,\pi$ are separably linked, by \cite[Lemma 11.2]{ChapmanGilatVishne} we have that $\Sigma_S$ is Witt equivalent to $\Sigma_{\{ \xi,\pi\} }$. Finally  by  \cref{leftoright} we have that $\Sigma_{\{ \xi,\pi\} }$ is Witt equivalent to $\langle \langle \gamma,\beta,\alpha]]$.
\end{proof}

\begin{thm}\label{mainthm}
If every set of three quaternion algebras over $F$ is linked, then $I_q^3 F=0$.
\end{thm}

\begin{proof}
Let $\varphi=\langle \langle \gamma,\beta,\alpha]]$ be an arbitrary quadratic 3-fold Pfister form. If $\beta=\alpha^2$ then $\varphi$ is hyperbolic. 
Otherwise, by  \cref{Witteq}, $\varphi$ is Witt equivalent to $\Sigma_S$ for some linked set $S$ of three quaternion algebras.
By the assumption, these three quaternion algebras share a common maximal subfield.
If this subfield is a separable field extension of $F$ then by \cref{lem:tight}, $(1)$ we have that   $\Sigma_G$ is Witt equivalent to a quadratic 4-fold Pfister form.
If this subfield is an inseparable field extension of $F$ then by \cref{lem:tight}, $(2)$ we have that  $\Sigma_G$ is hyperbolic.
In both cases, $\varphi$ must be hyperbolic, and so $I_q^3 F=0$.
\end{proof}

\begin{cor}\label{triplelinkageuinv}
If   $F$ is  linked  and  $I^3_q(F)=0$ then $u(F) \leq 4$. In particular, if every set of  three quaternion algebras over $F$ is linked then  $u(F) \leq 4$.
\end{cor}

\begin{proof}
If  $I_q^3 F=0$, by \cite[Lemma 4.3]{Chapman:2017} then there exists an anisotropic nonsingular quadratic form of dimension $u(F)$ whose Witt class is in $I_q^2 F$.
However, such a form is a direct sum of 2-fold Pfister forms (recall that when $I_q^3 F=0$, 2-fold Pfister forms are isometric to their similar forms), and since every two quaternion algebras over $F$ are linked, such a direct sum is Witt equivalent to one 2-fold Pfister form. Consequently, $u(F) \leq 4$. The second statement follows from \cref{mainthm}.
\end{proof}

\section{Characteristic different from 2}

Let $F$ be now a field of $\operatorname{char}(F) \neq 2$.
It is the case that  $\tilde{u}(F)= 4$ implies that every three quaternion algebras over $F$ are linked?
In this section we prove that this is indeed the case if the field is nonreal,  in which  case  $\tilde{u}(F)=u(F)$. Note that this question is trivial if $u(F)<4$, as  then there are no division quaternion algebras over $F$.

The main tool is the complex constructed in \cite{Peyre1995} and studied further in \cite{QueguinerTignol:2015}. For a finite subset $U$ of $Br(F)$, the Brauer group of $F$, we let $F^\times \cdot U$ denote the subgroup of $H^3(F,\mu_2)$ generated by classes $\lambda\cdot \alpha$ with $\lambda\in F^\times$ and $\alpha\in U$.
\begin{prop}[{\cite[Theorem 3.13]{QueguinerTignol:2015}}]\label{Peyre}
Let $S=\{Q_1,Q_2,Q_3\}$ be  a tight set  of quaternion algebras over $F$
and 
 $U$ be the subgroup $\langle Q_1,Q_2,Q_3 \rangle$ of the Brauer group $\prescript{}{2}Br(F)\cong\prescript{}{2} H^2(F,\mathbb{Q}/\mathbb{Z} (2))$. 
 Further, let $M$ be  the function field of the Cartesian product of the underlying Severi-Brauer varieties of $Q_1,Q_2,Q_3$. Then 
   the homology of the complex
$$\xymatrix{F^\times \cdot U \ar[r] & H^3(F,\mathbb{Q}/\mathbb{Z} (2)) \ar[r]^{\operatorname{res}}& H^3(M,\mathbb{Q}/\mathbb{Z} (2))}$$
has order 1 or 2.
It is of order 1 if and only if $S$ is linked.
\end{prop}

\begin{prop}\label{triplinkeasier}
Let $F$ be a field with $\operatorname{char}(F) \neq 2$ and $I^3_q(F)=0$.
Then every tight  set of  three quaternion algebras over $F$ is linked.
\end{prop}

\begin{proof}  
Let $S=\{Q_1,Q_2,Q_3\}$ be a tight set of quaternion algebras over $F$, $U$ be the subgroup of $\prescript{}{2}Br(F)$  generated by $S$  and  $M$ be the function field of the Cartesian product of the Severi-Brauer varieties of $Q_1,Q_2$ and $Q_3$.
Since $\prescript{}{2}H^3(F,\mathbb{Q}/\mathbb{Z} (2)) \cong I_q^3(F)=0$, the orders of elements in $H^3(F,\mathbb{Q}/\mathbb{Z} (2))$ must be odd or infinite.
By \cref{Peyre},  the order of the homology of the complex
$$\xymatrix{F^\times \cdot U \ar[r] & H^3(F,\mathbb{Q}/\mathbb{Z} (2)) \ar[r]^{\operatorname{res}}& H^3(M,\mathbb{Q}/\mathbb{Z} (2))}$$
is either $1$ or $2$.  In particular, the kernel of the restriction map cannot contain any elements of infinite order, as  elements of $F^\times \cdot U$ are of finite order. Therefore the order of the homology of the above complex must be $1$.
By  \cref{Peyre}, it follows that $S$ is linked.
\end{proof}

\begin{cor}\label{triplink}
Let $F$ be a nonreal field with $\operatorname{char}(F) \neq 2$ and ${u}(F)= 4$.
Then every set of three quaternion algebras over $F$ is linked.
\end{cor}

\begin{proof}  
Since ${u}(F)= 4$
we have that $F$ is linked by \cite[(39.1)]{EKM}. Hence for all triples $Q_1,Q_2$ and $Q_3$ of quaternion algebras over $F$, the set  $S=\{Q_1,Q_2,Q_3\}$  is tight. That $F$ is nonreal and ${u}(F)= 4$  implies that $I_q^3(F)=0$.  The result therefore follows from \cref{triplinkeasier}.
\end{proof}

\begin{ques}\label{end}
Let $F$ be a field with $\operatorname{char}(F) = 2$ and $u(F)= 4$.
Is every  set of three quaternion algebras over $F$ linked?
\end{ques}

Note that  if a similar result to \cref{Peyre} can be shown for fields of characteristic $2$, the same proof of \cref{triplink} could be used for these fields to give a positive answer to \cref{end}.

\section*{Acknowledgements} 
We  thank the referee for   useful suggestions that improved the clarity of the paper. 

The second author was supported by   
 \emph{Automorphism groups of locally finite trees} (G011012) with the Research Foundation, Flanders, Belgium (F.W.O.~Vlaanderen). %

\section*{Bibliography}
\bibliographystyle{amsalpha}

\end{document}